\documentclass[journal,onecolumn,draftclsnofoot]{IEEEtran}

\usepackage[OT1]{fontenc}
\usepackage{amsthm,amsmath,amssymb,subfigure,graphicx,cite}
\usepackage{color}
\usepackage[ruled]{algorithm2e}
\usepackage{hyperref}

\newtheorem{theorem}{Theorem}[section]
\newtheorem{lemma}[theorem]{Lemma}
\newtheorem{definition}[theorem]{Definition}

\def \E{\mathbb E}
\def \1{\mathbf 1}
\def \P{\mathrm{Pr}}

\def \S{\mathcal{S}}
\def \N{p}

\SetKw{KwInit}{Initialize:}

\SetKw{KwDistill}{Distillation:}

\SetKw{KwFinal}{Final Observations:}

\begin{document}

\title{Distilled Sensing: Adaptive Sampling for Sparse Detection and Estimation}

\author{Jarvis~Haupt, %~\IEEEmembership{Member,~IEEE,}
        Rui~Castro, %~\IEEEmembership{Fellow,~OSA,}
        and~Robert~Nowak%~\IEEEmembership{Life~Fellow,~IEEE}% <-this % stops a space
\thanks{J.H. is with the Department
of Electrical and Computer Engineering, Rice University; R.C.
is with the Department of Electrical Engineering, Columbia
University; R.N. is with the Department of Electrical and
Computer Engineering, University of Wisconsin---Madison.
E-mail: {\tt jdhaupt@rice.edu}, {\tt rmcastro@ee.columbia.edu},
{\tt nowak@ece.wisc.edu}.}%
\thanks{This work was supported in part by NSF Grant
CCF-0353079 and AFOSR Grant FA9550-09-1-0140, and is dedicated
to the memory of Dr. Dennis Healy, who inspired and supported
this direction of research in the context of the DARPA
Analog-to-Information Program. Dennis' guidance, vision, and
inspiration will be missed.}% <-this % stops a space
\thanks{Manuscript completed January 28, 2010;
revised May 27, 2010.}% <-this % stops a space
\thanks{A preliminary version of this report
will appear at the International Symposium on Information
Theory in Austin, TX in June 2010.}%
}

\maketitle

\begin{abstract}
  \sloppypar Adaptive sampling results in dramatic improvements
  in the recovery of sparse signals in white Gaussian noise.  A
  sequential adaptive sampling-and-refinement procedure called
  \emph{Distilled Sensing} (DS) is proposed and analyzed.  DS
  is a form of multi-stage experimental design and testing.
  Because of the adaptive nature of the data collection, DS can
  detect and localize far weaker signals than possible from
  non-adaptive measurements.  In particular, reliable detection
  and localization (support estimation) using non-adaptive
  samples is possible only if the signal amplitudes grow
  logarithmically with the problem dimension. Here it is shown
  that using adaptive sampling, reliable detection is possible
  provided the amplitude exceeds a constant, and localization
  is possible when the amplitude exceeds any arbitrarily slowly
  growing function of the dimension.
\end{abstract}

\section{Introduction}\label{sec:intro}

In high dimensional multiple hypothesis testing problems the
aim is to identify the subset of the hypotheses that differ
from the null distribution, or simply to decide if one or more
of the hypotheses do not follow the null.  There is now a well
developed theory and methodology for this problem, and the
fundamental limitations in the high dimensional setting are
quite clear.  However, most existing treatments of the problem
assume a non-adaptive measurement process.  The question of how
the limitations might differ under a more flexible, sequential
adaptive measurement process has not been addressed.  This
paper shows that this additional flexibility can yield
surprising and dramatic performance gains.

For concreteness let $x = (x_1,\dots,x_{\N})
\in\mathbb{R}^{\N}$ be an unknown sparse vector, such that most
(or all) of its components $x_i$ are equal to zero. The
locations of the non-zero components are arbitrary. This vector
is observed in additive white Gaussian noise and we consider
two problems:
\begin{description}
\item {\em Localization}: Infer the locations of the few non-zero components.
\item {\em Detection}: Decide whether $x$ is the all-zero
    vector.
\end{description}
Given a single, non-adaptive noisy measurement of $x$, a common
approach entails coordinate-wise thresholding of the observed
data at a given level, identifying the number and locations of
entries for which the corresponding observation exceeds the
threshold. In such settings there are sharp asymptotic
thresholds that the magnitude of the non-zero components must
exceed in order for the signal to be localizable and/or
detectable.  Such characterizations have been given in various
contexts in \cite{abramovich:06,donoho:06,genovese:09} for the
localization problem and \cite{ingster:97,
ingster:98,donoho:04} for the detection problem. A more
thorough review of these sorts of characterizations is given in
Section~\ref{sec:non_adaptive}.

In this paper we investigate these problems under a more
flexible measurement process.  Suppose we are able to
sequentially collect multiple noisy measurements of each
component of $x$, and that the data so obtained can be modeled
as
\begin{equation}\label{eqn:gen_model}
y_{i,j} \ = \  x_i \ + \ \gamma_{i,j}^{-1/2} \, w_{i,j}, \ i=1,\dots,{\N}, \ \ j=1,\dots,k\ .
\end{equation}
In the above a total of $k$ measurement steps is taken, $j$
indexes the measurement step, $w_{i,j} \overset{\rm
i.i.d.}{\sim} \mathcal{N}(0,1)$ are zero-mean Gaussian random
variables with unit variance, and $\gamma_{i,j}\geq 0$
quantifies the precision of each measurement.  When
$\gamma_{i,j} = 0$ we adopt the convention that component $x_i$
was not observed at step $j$.  The crucial feature of this
model is that it does not preclude sequentially adaptive
measurements, where the $\gamma_{i,j}$ can depend on past
observations $\{y_{i,\ell}\}_{i\in\{1,\ldots,\N\}, \ell < j}$.

In practice, the precision for a measurement at location $i$ at
step $j$ may be controlled, for example, by collecting multiple
independent samples and averaging to reduce the effective
observation noise, the result of which would be an observation
described by the model \eqref{eqn:gen_model}. In this case, the
parameters $\{\gamma_{i,j}\}$ can be thought of as proportional
to the number of samples collected at location $i$ at step $j$.
For exposure-based sampling modalities common in many imaging
scenarios, the precision parameters $\{\gamma_{i,j}\}$ can be
interpreted as being proportional to the length of time for
which the component at location $i$ is observed at step $j$.

In order to make fair comparisons to non-adaptive measurement
processes, the total precision budget is limited in the
following way.  Let $R({\N})$ be an increasing function of
${\N}$, the dimension of the problem (that is, the number of
hypotheses under scrutiny). The precision parameters
$\{\gamma_{i,j}\}$ are required to satisfy
\begin{equation}\label{eqn:budget}
\sum_{j=1}^k \sum_{i=1}^p \gamma_{i,j} \ \leq \ R({\N}) \, .
\end{equation}
For example, the usual non-adaptive, single measurement model
corresponds to taking $R({\N}) = {\N}$, $k =1$, and
$\gamma_{i,1}=1$ for $i=1,\dots,{\N}$.  This baseline can be
compared with adaptive procedures by keeping $R({\N})={\N}$,
but allowing $k>1$ and variables $\{\gamma_{i,j}\}$ satisfying
(\ref{eqn:budget}).

The multiple measurement process (\ref{eqn:gen_model}) is
applicable in many interesting and relevant scenarios.  For
example in gene association and expression studies, two-stage
approaches are gaining popularity (see
\cite{muller:07,zehetmayer:05,satagopan:03} and references
therein): in the first stage a large number of genes is
initially tested to identify a promising subset of them, and in
the second-stage these promising genes are subject to further
testing.  Such ideas have been extended to multiple-stage
approaches; see, for example \cite{zehetmayer:08}. Similar
two-stage approaches have also been examined in the signal
processing literature–--see \cite{bashan:08}. More broadly,
sequential experimental design has been popular in other fields
as well, such as in computer vision where it is known as
\emph{active vision} \cite{NSF:91}, or in machine learning,
where it is known as \emph{active learning}
\cite{cohn:94,cohn:96}. These types of procedures can
potentially impact other areas such as microarray-based studies
and astronomical surveying.

The main contribution of this paper is a theoretical analysis
that reveals the dramatic gains that can be attained using such
sequential procedures. Our focus here is on a particular
sequential, adaptive sampling procedure called \emph{Distilled
Sensing} (DS). The idea behind DS is simple: use a portion of
the precision budget to crudely measure all components;
eliminate a fraction of the components that appear least
promising from further consideration after this measurement;
and iterate this procedure several times, at each step
measuring only components retained after the previous step. As
mentioned above, similar procedures have been proposed in the
context of experimental design, however to the best of our
knowledge the quantification of performance gains had not been
established prior to our own initial work in \cite{haupt:08b,
haupt:09a} and the results established in this paper. In this
manuscript we significantly extend our previous results by
providing stronger results for the localization problem, and an
entirely novel characterization of the detection problem.

This paper is organized as follows. Following a brief review of
the fundamental limits of non-adaptive sampling for detection
and localization in Section~\ref{sec:non_adaptive}, our main
result---that DS can reliably solve the localization and
detection problems for dramatically weaker signals than what is
possible using non-adaptive measurements---is stated in
Section~\ref{sec:ds_result}. A proof of the main result is
given in Section~\ref{sec:ds_proof}. Simulation results
demonstrating the theory are provided in Section~\ref{sec:exp},
and conclusions and extensions are discussed in
Section~\ref{sec:conclusion}. A proof of the threshold for
localization from non-adaptive measurements and several
auxiliary lemmas are provided in the appendices.

\section{Review of Non-adaptive Localization and Detection of Sparse
Signals}\label{sec:non_adaptive}

In this section we review the known thresholds for localization
and detection from non-adaptive measurements. As mentioned
above, such thresholds have been established in a variety of
problem settings \cite{abramovich:06,donoho:06,
genovese:09,ingster:97,ingster:98, donoho:04}. Here we provide
a concise summary of the main ideas along with supporting
proofs as needed, to facilitate comparison with our main
results concerning recovery from adaptive measurements which
appear in the next section.

The non-adaptive measurement model we will consider as the
baseline for comparison is as follows.  We have a single
observation of $x$ in noise:
\begin{equation}\label{eqn:na_model}
  y_i \ = \  x_i \ + \  w_i, \ i=1,\dots,{\N} \, ,
\end{equation}
where $w_i \overset{\rm i.i.d.}\sim {\cal N}(0,1)$.  As noted
above, this is a special case of our general setup
(\ref{eqn:gen_model}) in which $k=1$ and $\gamma_{i,1} = 1$ for
$i=1,\dots,{\N}$.  This implies a precision budget $R({\N}) =
\sum_{i=1}^{\N} \gamma_{i,1} = {\N}$.

To describe the asymptotic (large ${\N}$) thresholds for
localization we need to introduce some notation.  Define the
{\em false-discovery proportion} (FDP) and {\em non-discovery
proportion} (NDP) as follows.
\begin{definition}
  Let ${\cal S} := \{i \, : \, x_i \neq 0 \}$ denote the signal support set and let
  $\widehat{\cal S}=\widehat{\cal S}(y)$ denote an estimator of
  ${\cal S}$.  The false-discovery proportion is
  \begin{equation*}\label{eqn:fdp}
    \mbox{\rm FDP}(\widehat{\cal S}) \ := \
    \frac{|\widehat{\cal S} \backslash {\cal S}|}{|\widehat{\cal S}|} \ .
  \end{equation*}
  In words, the $\mbox{\rm FDP}$ of $\widehat{\cal S}$ is the
  ratio of the number of components falsely declared as
  non-zero to the total number of components declared non-zero.
  The non-discovery proportion is
  \begin{equation*}\label{eqn:ndp}
    \mbox{\rm NDP}(\widehat{\cal S}) \ := \ \frac{|S \backslash \widehat{\cal S}|}{|S|} \ .
  \end{equation*}
  In words, the $\mbox{\rm NDP}$ of $\widehat{\cal S}$ is the
  ratio of the number of non-zero components missed to the
  number of actual non-zero components.
\end{definition}

In this paper we focus in particular on the scenario where
$x_i\geq 0$ for all $i\in\{1,\ldots,\N\}$. We elaborate on
possible extensions in Section~\ref{sec:ds_result}. Under this
assumption it is quite natural to focus on a specific class of
estimators of ${\cal S}$.
\begin{definition}
  A {\em coordinate-wise thresholding procedure} is an
  estimator of the following form:
  \begin{equation*}
    \widehat{\cal S}_\tau(y) \ := \ \left\{i \in \{1,\dots,{\N}\}: y_i \geq \tau >0\right\} \ ,
  \end{equation*}
  where the threshold $\tau$ may depend implicitly on $x$, or
  on $y$ itself.
\end{definition}

The following result establishes the limits of localization
using non-adaptive sampling. A proof is provided in
Appendix~\ref{sec:naloc}.
\begin{theorem}\label{thm:naloc}
  Assume $x$ has ${\N}^{1-\beta}$, $\beta \in (0,1)$,
  non-zero components of amplitude $\sqrt{2r\log {\N}}$, $r>0$,
  and measurement model (\ref{eqn:na_model}).  There exists a
  coordinate-wise thresholding procedure that yields an
  estimator $\widehat{\cal S}=\widehat{\cal S}(y)$ such that if $r>\beta$, then
  as $\N \rightarrow\infty$,
  \begin{eqnarray*}
      \mbox{\rm FDP}(\widehat{\cal S})  & \overset{P}\rightarrow &
      0 \, ,  \ \ \
      \mbox{\rm NDP}(\widehat{\cal S}) \ \overset{P}\rightarrow\
      0 \ ,
  \end{eqnarray*}
  where $\stackrel{P}{\rightarrow}$ denotes convergence in
  probability.  Moreover, if $r < \beta$, then there does not
  exist a coordinate-wise thresholding procedure that can
  guarantee that both quantities above tend to $0$ as
  ${\N}\rightarrow \infty$.
\end{theorem}
We also refer the reader to recent related work in
\cite{genovese:09}, which considered localization under similar
error metrics as those utilized here. There it was shown, using
a random signal model and assuming observations in the form of
noisy independent random (Gaussian) linear combinations of the
entries of $x$, that similar sharp asymptotics hold for any
recovery procedure \cite[Thm.~5]{genovese:09}.

Random signal models have also been adopted in the examination
of the fundamental limits of signal detection \cite{ingster:97,
ingster:98, donoho:04}. In particular, suppose that $x$ is such
that its entries $x_i$ have amplitude $\mu(p)=\sqrt{2 r\log
{\N}}$ independently with probability
$\theta(p)={\N}^{-\beta}$, and amplitude zero with probability
$1-\theta(p)$. The problem of signal detection from noisy
observations collected according to the measurement model
\eqref{eqn:na_model} amounts to a hypothesis test of the form:
\begin{eqnarray} \label{eqn:det}
  {\rm H}_0 & : & y_i \ \overset{\rm iid}{\sim} \ {\cal N}(0,1), \ i=1,\dots,{\N} \nonumber \\
  {\rm H}_1 & : & y_i \ \overset{\rm iid}{\sim} \ (1-\theta({\N})) \,
  {\cal N}(0,1)+\theta({\N}) \, {\cal N}(\mu({\N}),1), \
  i=1,\dots,{\N}
\end{eqnarray}
Note that under the alternative hypothesis, the signal has
${\N}^{1-\beta}$ non-zero components in expectation. We recall
the following result \cite{ingster:97,ingster:98,donoho:04}.
\begin{theorem} \label{thm:nadet}
  Consider the hypotheses in (\ref{eqn:det}) where $\mu(\N)=\sqrt{2r\log \N}$. Define
  \begin{eqnarray*}\label{eqn:hct}
    \rho(\beta) & := & \left\{\begin{array}{ll}
    0, & 0 < \beta \leq 1/2 \\
    \ \vspace{-.1in}  \\
    \beta-1/2, & 1/2 < \beta \leq 3/4 \\
    \ \vspace{-.1in}  \\
    (1-\sqrt{1-\beta})^2, & 3/4 < \beta < 1
    \end{array}\right.
  \end{eqnarray*}
  If $r>\rho(\beta)$, then there exists a test for which the
  sum of the false alarm and miss probabilities
  tends to $0$ as ${\N}\rightarrow \infty$. Conversely, if $r <\rho(\beta)$,
  then for any test the sum of the false alarm and miss probabilities
  tends to $1$ as ${\N}\rightarrow \infty$.
\end{theorem}
It is possible to relate these detection results to the
deterministic sparsity model that we consider here, using the
ideas presented in \cite[Chapter 8]{ingster:03}.

\section{Main Results: Adaptive Localization and Detection of Sparse
Signals}\label{sec:ds_result}

In this section we present the main results of our theoretical
analysis of Distilled Sensing (DS). Algorithm~\ref{alg:DS}
describes the DS measurement process. At each step of the
process, we retain only the components with non-negative
observations.  This means that when the number of non-zero
components is very small, roughly half of the components are
eliminated from further consideration at each step.
Consequently, if the precision budget allocated at each step is
slightly larger than $1/2$ of that used in the preceding step,
then the effective precision of the measurements made at each
step is increasing. In particular, if the budget for each step
is $1/2+c$ of the budget at the previous step, for some small
constant $c>0$, then the precision of the measured components
is increasing exponentially. Therefore, the key is to show that
the very crude thresholding at $0$ at each step does not remove
a significant number of the non-zero components.  One final
observation is that because the number of components measured
decreases by a factor of roughly $1/2$ at each step, the total
number of measurements made by DS is roughly $2{\N}$, a modest
increase relative to the ${\N}$ measurements made in the
non-adaptive setting.

Recall from above that for non-adaptive sampling, reliable
detection and localization is only possible provided the signal
amplitude is $\Omega(\sqrt{\log(\N)})$. In other words, the
signal amplitude must exceed a constant (that depends on the
sparsity level) times $\sqrt{\log(\N)}$. The following theorem
establishes that DS is capable of detecting and localizing much
weaker sparse signals. For the purposes of our investigation we
assume that the non-zero components are positive.  It is
trivial to extend the algorithm and its analysis to handle both
positive and negative components by simply repeating the entire
process twice; once as described, and again with $y_{i,j}$
replaced with $-y_{i,j}$ in the refinement step of
Algorithm~\ref{alg:DS}.

%\IncMargin{1em}
\begin{algorithm}[t]
\vspace{.05in} \KwIn{} \Indp Number of observation steps: $k$\;

Resource allocation sequence satisfying $\sum_{j=1}^k R_{j}
  \leq R({\N})$\; \Indm
%  \BlankLine
  \BlankLine
  \KwInit{}\\
  \Indp

  Initial index set: $I_{1} \longleftarrow
  \{1,2,\dots,{\N}\}$\; \Indm
%  \BlankLine
  \BlankLine
  \KwDistill{}\\
  \Indp
  \SetAlgoNoLine
  \For{$j = 1$ \emph{\KwTo} $k$}{ Allocate resources:
  $\gamma_{i,j} =
    \left\{\begin{array}{cl} R_{j}/|I_{j}| &
    i\in I_{j}\\
    0 & i \notin I_{j} \end{array}\right\}$\;
    \ \\
    \ \\
    Observe: $y_{i,j} = x_i \, + \gamma_{i,j}^{-1/2}w_{i,j}, \
  i \in I_{j}$\;
    \ \\
    \ \\
    Refine: $I_{j+1} \longleftarrow \{i\in I_{j}: y_{i,j}
    >0 \}$\; }

    \Indm
  %\BlankLine
  \BlankLine
  \KwOut{}
  \Indp

  Final index set: $I_{k}$\;

  Distilled observations: $y_{k} = \{y_{i,k}: i\in I_{k}\}$\;

  \Indm

  \caption{Distilled Sensing.}
  \label{alg:DS}
\end{algorithm}

\begin{theorem}\label{thm:ds}
  Assume $x\geq0$ with ${\N}^{1-\beta}$, $\beta \in (0,1)$,
  non-zero components of amplitude $\mu({\N})$, and sequential
  measurement model using Distilled Sensing with $k = k({\N}) =
  \max\{\lceil \log_2\log {\N} \rceil,0\}+2$, and precision
  budget distributed over the measurement steps so that
  $\sum_{j=1}^k R_j \leq {\N}$, $R_{j+1}/R_j \geq \delta >
  1/2$, and $R_1=c_1 {\N}$ and $R_k = c_k \,{\N}$ for some
  $c_1,c_k \in (0,1)$. Then the support set estimator
  constructed using the output of the DS algorithm
  \begin{eqnarray*}
    \widehat{\cal S}_{\rm DS} & := & \{i\in I_{k}: y_{i,k}>\sqrt{2/c_k}\}
  \end{eqnarray*}
  has the following properties:
  \begin{description}
    \item (i) if $\mu({\N})\rightarrow \infty$ as a
        function of $\N$, then as ${\N}\rightarrow \infty$
    \begin{eqnarray*}
      \mbox{\rm FDP}(\widehat{\cal S}_{\rm DS})  & \overset{P}\rightarrow &
      0 \, ,  \ \ \
      \mbox{\rm NDP}(\widehat{\cal S}_{\rm DS}) \ \overset{P}\rightarrow\
      0 \ ,
    \end{eqnarray*}
    \item (ii) if $\mu({\N}) >
        \max\left\{\sqrt{4/c_1},2\sqrt{2/c_k}\right\}$ (a
        constant) then
    \begin{eqnarray*}
      \lim_{{\N}\rightarrow \infty} \P(\widehat{\cal S}_{\rm DS} = \emptyset)
      & = & \left\{\begin{array}{cc} 1\, , & \mbox{if $x=0$} \\
      0 \, , & \mbox{if $x\neq 0$}
      \end{array}\right.,
    \end{eqnarray*}
    where $\emptyset$ is the empty set.
  \end{description}
\end{theorem}

In words, this result states that DS successfully identifies
the sparse signal support provided only that the signal
amplitude grows (arbitrarily slowly) as a function of the
problem dimension $\N$, while reliable signal detection
requires only that the signal amplitude exceed a constant. The
result (ii) is entirely novel, and (i) improves on our initial
result in \cite{haupt:09a} which required $\mu(\N)$ to grow
faster than an arbitrary iteration of the logarithm (i.e.,
$\mu(\N) \sim \log \log \dots \log \N$). Comparison with the
$\Omega(\sqrt{\log{\N}})$ amplitude required for both tasks
using non-adaptive sampling illustrates the dramatic gains that
are achieved through adaptivity.

\section{Analysis of Distilled Sensing}\label{sec:ds_proof}

In this section we prove the main result characterizing the
performance of Distilled Sensing (DS), Theorem~\ref{thm:ds}. We
begin with three lemmas that quantify the finite sample
behavior of DS.

\subsection{Distillation: Reject the Nulls, Retain the Signal}

\begin{lemma}\label{lem:null_keep}
If $\{y_i\}_{i=1}^m \overset{\rm iid}\sim \mathcal{N}(0,\sigma^2)$, $\sigma>0$, then
for any $0  < \varepsilon < 1/2$,
  \begin{equation*}
    \left(\frac{1}{2}-\varepsilon\right)m \leq \Big|
    \{i\in\{1,\dots,m\}: y_i > 0\}\Big| \leq
    \left(\frac{1}{2}+\varepsilon\right)m,
  \end{equation*}
  with probability at least $1-2\exp{(-2m \varepsilon^2)}$.
\end{lemma}

\begin{proof}
 For any event $A$, let $\1_{A}$ be the indicator taking the value $1$ if
 $A$ is true and $0$ otherwise. By Hoeffding's inequality, for any $\varepsilon >0$
  \begin{equation*}
    \P\left(\left|\sum_{i=1}^m \1_{\{y_i> 0\}} - \frac{m}{2}\right| >
    m\varepsilon\right)\leq
    2\exp{\left(-2 m \varepsilon^2\right)}.
  \end{equation*}
Imposing the restriction
  $\varepsilon<1/2$ guarantees that the corresponding fractions
  are bounded away from zero and one.
\end{proof}

\begin{lemma}\label{lem:sig_keep}
Let $\{y_i\}_{i=1}^m \overset{\rm iid}\sim \mathcal{N}(\mu,\sigma^2)$, with
$\sigma>0$ and $\mu\geq 2\sigma$. Define $\epsilon = \frac{\sigma}{\mu \sqrt{2 \pi}}< 1$. Then
  \begin{equation*}
    (1-\epsilon)m \leq \Big|\{i\in\{1,2,\dots,m\}:y_i>0\}\Big| \leq
    m,
  \end{equation*}
  with probability at least $1-\exp{\left(-\frac{\mu m}{4 \sigma\sqrt{2 \pi}}\right)}.$
\end{lemma}

\begin{proof}
  We will utilize the following standard bound on the
  Gaussian tail: for $Z\sim\mathcal{N}(0,1)$ and
  $\gamma>0$,
  \begin{equation*}
    \frac{1}{\sqrt{2\pi\gamma^2}}\left(1-\frac{1}{\gamma^2}\right)
    \exp(-\gamma^2/2) \leq \Pr(Z> \gamma) \leq
    \frac{1}{\sqrt{2\pi\gamma^2}} \exp(-\gamma^2/2).
  \end{equation*}
  Let $q = \Pr(y_i >0)$, then it follows that
  \begin{equation*}
    1-q \ \leq \ \frac{\sigma}{\mu\sqrt{2\pi}} \exp{\left(-\frac{\mu^2}{2\sigma^2}\right)}\ .
  \end{equation*}
  Next we use the Binomial tail bound from \cite{chernoff:52}:
  for any $0<b < \E[ \sum_{i=1}^m \1_{\{y_i >0\}}] = mp$,
  \begin{equation*}
    \P\left(\sum_{i=1}^m \1_{\{y_i >0\}} \leq b\right) \leq \left(\frac{m-mp}{m-b}\right)^{m-b}
    \left(\frac{mp}{b}\right)^{b}\ .
  \end{equation*}
  Note that $\epsilon > 1-q$ (or equivalently, $1-\epsilon <
  q$), so we can apply this result to $\sum_{i=1}^m \1_{\{y_i
  >0\}}$ with $b=(1-\epsilon)m$ to obtain
  \begin{eqnarray*}
    \P\left(\sum_{i=1}^m \1_{\{y_i >0\}} \leq (1-\epsilon)m\right) &\leq& \left(\frac{1-q}{\epsilon}\right)^{\epsilon m}
    \left(\frac{q}{1-\epsilon}\right)^{(1-\epsilon)m}\\
    &\leq&\exp{\left(-\frac{\mu^2\epsilon m}{2\sigma^2}\right)} \left(\frac{1}{1-\epsilon}\right)^{(1-\epsilon)m}.
  \end{eqnarray*}
  Now, to establish the stated result, it suffices to show
  \begin{eqnarray*}
    -\frac{\mu^2}{2\sigma^2} + \left(\frac{1-\epsilon}{\epsilon}\right) \log{\left(\frac{1}{1-\epsilon}\right)} &\leq&
    -\frac{\mu}{4\epsilon\sigma\sqrt{2\pi}} \ = \ -\frac{\mu^2}{4\sigma^2},
  \end{eqnarray*}
  which holds provided $\mu \geq 2\sigma$, since $0<\epsilon<1$
  and $\left(\frac{1-\epsilon}{\epsilon}\right)
  \log{\left(\frac{1}{1-\epsilon}\right)}  \leq 1$ for
  $\epsilon\in (0,1)$.
\end{proof}

\subsection{The Output of the DS Procedure}

Refer to Algorithm~\ref{alg:DS} and  define  $s_j := |\S
\bigcap I_j|$ and $z_j := |\S^c\bigcap I_j|$, the number of
non-zero and zero components, respectively, present at the
beginning of step $j$, for $j=1,\dots,k$. Let $\varepsilon >
0$, and for $j=1,\dots,k-1$ define
\begin{eqnarray}
  %\epsilon_j &:= &
  %\sqrt{\frac{s_1+(1/2+\varepsilon)^{j-1} z_1}{2\pi \mu^2 R_j}} \ ,
  \epsilon_j^2 &:= &
  {\frac{s_1+(1/2+\varepsilon)^{j-1} z_1}{2\pi \mu^2 R_j}} \ ,\label{eqn:epsilons}
\end{eqnarray}
The output of the DS procedure is quantified in the following
result.

\begin{lemma}\label{lem:gen_result}
  Let  $0 < \varepsilon < 1/2$ and assume that $R_j >
  \frac{4}{\mu^2} \left(s_1+(1/2+\varepsilon)^{j-1}z_1\right)$,
  $j=1,\dots,k-1$.  If $|\S|>0$, then with probability at least
  \begin{equation*}
    1-\sum_{j=1}^{k-1}\exp{\left(\frac{-s_1 \prod_{\ell=1}^{j-1}
    (1-\epsilon_\ell)}{2\sqrt{2\pi}}\right)} -2\sum_{j=1}^{k-1}
    \exp{(-2 z_1 (1/2-\varepsilon)^{j-1} \varepsilon^2)} \ ,
  \end{equation*}
  $\prod_{\ell=1}^{j-1} (1-\epsilon_{\ell})s_{1} \leq s_{j}
  \leq s_1$ and $\left(\frac{1}{2}-\varepsilon\right)^{j-1} z_1
  \leq z_{j} \leq
  \left(\frac{1}{2}+\varepsilon\right)^{j-1}z_1$ for
  $j=2,\dots,k$.  If $|\S| = 0$, then with probability at least
  \begin{equation*}
    1-2\sum_{j=1}^{k-1}\exp{(-2 z_1 (1/2-\varepsilon)^{j-1} \varepsilon^2)} \ ,
  \end{equation*}
  $\left(\frac{1}{2}-\varepsilon\right)^{j-1} z_1 \leq z_{j} \leq
  \left(\frac{1}{2}+\varepsilon\right)^{j-1}z_1$ for $j=2,\dots,k$.
\end{lemma}

\begin{proof}
  The results follow from
  Lemmas~\ref{lem:null_keep}~and~\ref{lem:sig_keep} and the
  union bound. First assume that $s_1 = |\S|>0$.   Let
  $\sigma^2_j := |I_j|/R_{j}= (s_j+z_j)/R_j$ and
  $\widetilde\epsilon_j := \frac{\sigma_j}{\mu \sqrt{2 \pi}}$,
  $j=1,\dots,k$.

  The argument proceeds by conditioning on the output of all
  prior refinement steps; in particular, suppose that
  $(1-\widetilde\epsilon_{\ell-1})s_{\ell-1} \leq s_{\ell} \leq
  s_{\ell-1}$ and
  $\left(\frac{1}{2}-\varepsilon\right)z_{\ell-1} \leq z_{\ell}
  \leq \left(\frac{1}{2}+\varepsilon\right)z_{\ell-1}$for $\ell
  = 1,\dots,j$. Then apply Lemma~\ref{lem:null_keep} with
  $m=z_{j}$, Lemma~\ref{lem:sig_keep} with $m=s_j$ and
  $\sigma^2 = \sigma_j^2$, and the union bound to obtain that
  with probability at least
  \begin{equation}
    1-\exp{\left(-\frac{\mu s_j}{4 \sigma_{j}\sqrt{2
    \pi}}\right)} -2\exp{(-2 z_j \varepsilon^2)} \ ,\label{eq:pb}
  \end{equation}
  $(1-\widetilde\epsilon_{j})s_{j} \leq s_{j+1} \leq s_{j}$,
  and $\left(\frac{1}{2}-\varepsilon\right)z_j \leq z_{j+1}
  \leq \left(\frac{1}{2}+\varepsilon\right)z_j$.  Note that the
  condition $R_j > \frac{4}{\mu^2}
  \left(s_1+(1/2+\varepsilon)^{j-1}z_1\right)$ and the
  assumptions on prior refinement steps ensure that $\mu> 2 \,
  \sigma_j$, which is required for Lemma~\ref{lem:sig_keep}.
  The condition $\mu> 2 \, \sigma_j$ also allows us to simplify
  probability bound (\ref{eq:pb}), so that the event above
  occurs with probability at least
  \begin{equation*}
    1-\exp{\left(-\frac{s_j}{2\sqrt{2
    \pi}}\right)} -2\exp{(-2 z_j \varepsilon^2)}.
  \end{equation*}

  Next, we can recursively apply the union bound and the bounds
  on $s_j$ and  $z_j$ above to obtain for
  $j=1,\dots,k-1$
  \begin{eqnarray*}
    \epsilon_j & =  & \sqrt{\frac{s_1+(1/2+\varepsilon)^{j-1} z_1}{2\pi \mu^2
    R_j}}  \ \geq \ \widetilde \epsilon_j \ = \ \frac{\sigma_j}{\mu \sqrt{2\pi}} \ ,
  \end{eqnarray*}
  with probability at least
  \begin{equation*}
    1-\sum_{j=1}^{k-1}\exp{\left(\frac{-s_1 \prod_{\ell=1}^{j-1}
    (1-\epsilon_\ell)}{2\sqrt{2 \pi}}\right)} -\sum_{j=1}^{k-1} 2\exp{(-2 z_1
    (1/2-\varepsilon)^{j-1} \varepsilon^2)} \ .
  \end{equation*}
  Note that the condition $R_j > \frac{4}{\mu^2}
  \left(s_1+(1/2+\varepsilon)^{j-1}z_1\right)$ implies that
  $\epsilon_j < 1$. The first result follows directly. If $s_1
  = |\S| = 0$, then consider only $z_j$, $j=1,\dots,k$. The
  result follows again by the union bound. Note that for this
  statement the condition on $R_j$ is not required.
\end{proof}

\sloppypar Now we examine the conditions $R_j > \frac{4}{
\mu^2} \left(s_1+(1/2+\varepsilon)^{j-1}z_1\right)$,
$j=1,\dots,k$ more closely. Define
$c:=s_1/[(1/2+\varepsilon)^{k-1} z_1]$, in effect condensing
several problem-specific parameters ($s_1$, $z_1$, and $k$)
into a single scalar parameter. Then the conditions on $R_j$
are satisfied if
\begin{equation*}
  R_j >\frac{4z_1 (1/2+\varepsilon)^{j-1}}{\mu^2} (c
  (1/2+\varepsilon)^{k-j}+1) \ .
\end{equation*}
Since $z_1 \leq {\N}$, the following condition is sufficient
\begin{equation*}
  R_j > \frac{4{\N} (1/2+\varepsilon)^{j-1}}{\mu^2} (c (1/2+\varepsilon)^{k-j}+1) \ ,
\end{equation*}
and in particular the more stringent condition \mbox{$R_j >
\frac{4(c+1){\N} (1/2+\varepsilon)^{j-1}}{\mu^2}$} will
suffice. It is now easy to see that if $s_1 \ll z1$ (e.g., so
that $c \leq 1$), then the sufficient conditions become $R_j >
\frac{8{\N}}{ \mu^2} (1/2+\varepsilon)^{j-1}$, $j=1,\dots,k$.
Thus, for the sparse situations we consider, the precision
allocated to each step must be just slightly greater than $1/2$
of the precision allocated in the previous step. We are now in
position to prove the main theorem.

\subsection{Proof of Theorem~III.1}

\sloppypar Throughout the proof, whenever asymptotic notation
or limits are used it is always under the assumption that
${\N}\rightarrow\infty$, and we use the standard notation
$f({\N})=o(g({\N}))$ to indicate that $\lim_{{\N} \rightarrow
\infty} f({\N})/g({\N}) = 0$, for $f({\N}) \geq 0$ and $g({\N})
> 0$. Also the quantities $k:=k({\N})$,
$\varepsilon:= \varepsilon({\N})$ and $\mu:=\mu({\N})$ are
functions of ${\N}$, but we do not denote this explicitly for
ease of notation. We let $\varepsilon :={\N}^{-1/3}$ throughout
the proof.

We begin by proving part (ii) of the theorem, which is
concerned with detecting the presence or absence of a sparse
signal. Part (i), which pertains to identifying the locations
of the non-zero components, then follows with a slight
modification.

\textbf{Case 1 -- Signal absent ($\mathcal{S}=\emptyset$):}
This is the simplest scenario, but through its analysis we will
develop tools that will be useful when analyzing the case where
the signal is present. Here, we have $s_1=0$ and $z_1={\N}$,
and the number of indices retained at the end of the DS
procedure $|I_k|$ is equal to $z_k$. Define the event
\begin{equation*}
  \Gamma \ = \ \left\{\left(\frac{1}{2}-
  \varepsilon\right)^{k-1}{\N} \ \leq \ |I_{k}| \ \leq \
  \left(\frac{1}{2}+\varepsilon\right)^{k-1}{\N}\right\} \ .
\end{equation*}
The second part of Lemma~\ref{lem:gen_result} characterizes the
probability of this event; in particular
\begin{equation*}
  \Pr(\Gamma) \ \geq \ 1-2\sum_{j=1}^{k-1}\exp\left(-2{\N}
  \left(\frac{1}{2}-\varepsilon\right)^{j-1} \varepsilon^2 \right) .
\end{equation*}
Since $k\leq \log_2\log {\N}+3$, for large enough $\N$ we get
that
\begin{eqnarray*}
  \Pr(\Gamma) &\geq& 1-2(k-1)\exp\left(-2{\N} \left(\frac{1}{2}-\varepsilon\right)^{k-2} \varepsilon^2 \right) \\
  &=& 1-2(k-1)\exp\left(-{\N} \left(\frac{1}{2}\right)^{k-3} (1-2\varepsilon)^{k-2} \varepsilon^2 \right) \\
  &\geq& 1-2(\log_2 \log {\N}+2)\exp\left(-\frac{{\N}^{1/3}}{\log {\N}}(1-o(1)) \right)
\end{eqnarray*}
where we used Lemma~\ref{lem:lim} to conclude that
$(1-2\varepsilon)^{k-2} = 1-o(1)$. It is clear that
$\Pr(\Gamma)\rightarrow 1$.

In this case we assume that $\mathcal{S}=\emptyset$, therefore
the output of the DS procedure consists of $|I_k|$ i.i.d.
Gaussian random variables with zero mean and variance
$|I_k|/R_k=|I_k|/(c_k {\N})$. Note that given $\Gamma$,
\begin{eqnarray*}
  |I_k| \ \leq \
  {\N}\left(\frac{1}{2}+\varepsilon\right)^{k-1}&=&
  {\N}\frac{1}{2}\left(\frac{1}{2}\right)^{k-2}\left(1+2\varepsilon\right)^{k-1}\\
  &\leq& \frac{1}{2} \frac{{\N}}{\log {\N}} (1+o(1))\ ,
\end{eqnarray*}
which follows from the fact that $k\geq \log_2\log {\N}+2$, and
using Lemma~\ref{lem:lim}. With this in hand we conclude that
(with a slight abuse of notation)
\begin{eqnarray*}
  \Pr(\widehat{\mathcal{S}}_{\rm DS}\neq \emptyset \ | \ \Gamma) &=& \Pr\left(\exists_{i\in I_k} : y_{i,k}>
  \sqrt{2/c_k}\right)\\
  &\leq& |I_k| \Pr \left(\mathcal{N}(0,|I_k|/c_k {\N})>\sqrt{2/c_k}\right)\\
  &=& |I_k| \Pr \left(\mathcal{N}(0,1)>\sqrt{2{\N}/|I_k|}\right)\\
  &\leq& {\N} \Pr \left(\mathcal{N}(0,1)>\sqrt{4\log {\N} (1-o(1))}\right)\\
  &\leq& {\N} \exp\left(-2\log {\N} (1-o(1))\right)\\
  &=& {\N}^{-1+o(1)} \rightarrow 0\ ,
\end{eqnarray*}
where the last inequality follows from the standard Gaussian
tail bound. This together with $\Pr(\Gamma)\rightarrow 1$
immediately shows that when $\mathcal{S}=\emptyset$ we have
$\Pr(\widehat{\mathcal{S}}_{\rm DS}\neq \emptyset)\rightarrow
0$.

\textbf{Case 2 -- Signal present ($\mathcal{S}\neq\emptyset$):}
The proof follows the same idea as in the previous case,
although the argument is a little more involved. Begin by
applying Lemma~\ref{lem:gen_result} and constructing an event
that occurs with probability tending to one. Let $\Gamma$ be
the event
\begin{eqnarray*}
  \Gamma & = & \left\{z_1\left(\frac{1}{2}-\varepsilon\right)^{k-1} \ \leq \ z_k \ \leq \
  z_1\left(\frac{1}{2}+\varepsilon\right)^{k-1} \right\} \\
  & & \quad \quad \bigcap \quad \quad \left\{ s_1\prod_{j=1}^{k-1} (1-\epsilon_j) \ \leq  \ s_k \ \leq \ s_1 \right\} \ ,
\end{eqnarray*}
where $\epsilon_j$ is given by equation \eqref{eqn:epsilons}.
Lemma~\ref{lem:gen_result} characterizes the probability of
this event under a condition on $R_j$ that we will now verify.
Note that this condition is equivalent to $\epsilon_j^2<
1/(8\pi)$ for all $j=1,\ldots,k-1$. Instead of showing exactly
this we will show a stronger result that will be quite useful
in a later stage of the proof. Recall that $R_{j+1}/R_{j}\geq
\delta>1/2$, $j=1,\ldots,k-2$, and $R_1=c_1 {\N}$ by the
assumptions of the theorem. Thus for $j=1,\ldots,k-1$
\begin{eqnarray*}
  \epsilon^2_j &\leq& \frac{s_1+\left(\frac{1}{2}+\varepsilon\right)^{j-1}z_1}{2\pi\mu^2 \delta^{j-1} R_1} \\
  &\leq& \frac{1}{2\pi\mu^2 c_1}\left(\frac{s_1}{{\N}}\delta^{-(j-1)}+\frac{z_1}{{\N}}
  \left(\frac{\delta}{\frac{1}{2}+\varepsilon}\right)^{-(j-1)} \right)\ .
\end{eqnarray*}
Clearly we have that $\epsilon^2_1\leq\frac{1}{2\pi\mu^2
c_1}<1/(8\pi)$ since by assumption $\mu>\sqrt{4/c_1}$.  Now
consider the case $j>1$. Recall that $k\leq \log_2\log {\N}
+3$. Therefore if $\delta\geq 1$, then the term
$\delta^{-(j-1)}$ can be upper bounded by 1, otherwise
\begin{equation}
  \delta^{-(j-1)}\leq \delta^{-(k-2)} \leq \delta^{-(\log_2 \log {\N}+1)}
  =\delta^{-1} \left(\log {\N}\right)^{-\log_2 \delta} \leq 2\log
  {\N}\ ,\label{eqn:power_of_rho}
\end{equation}
where the last step follows from $\delta>1/2$.

Now recall that $s_1={\N}^{1-\beta}$, therefore
\begin{eqnarray}
  \epsilon^2_j &\leq& \frac{1}{2\pi\mu^2
  c_1}\left({\N}^{-\beta}\delta^{-(j-1)}+\left(\frac{\delta}{\frac{1}{2}+\varepsilon}\right)^{-(j-1)} \right)\nonumber\\
  &\leq& \frac{1}{2\pi\mu^2 c_1}\left(2 {\N}^{-\beta}\log {\N}+\left(\frac{\delta}{\frac{1}{2}+\varepsilon}\right)^{-(j-1)}
  \right)\ .\label{eqn:epsilon_bound}
\end{eqnarray}
Note that, since $\varepsilon\rightarrow 0$ as ${\N}\rightarrow
\infty$ we have that, for ${\N}$ large enough,
$\delta/(1/2+\varepsilon)>(\delta+1/2+\varepsilon)$. Assume
${\N}$ is large enough so that this is true, then
\begin{eqnarray*}
  \epsilon^2_j & \leq & \frac{1}{2\pi\mu^2 c_1}\left(2{\N}^{-\beta}\log {\N}
  +\left(\delta+\frac{1}{2}+\varepsilon\right)^{-(j-1)} \right)\ .
\end{eqnarray*}
Clearly since $j\leq k-1 \leq \log_2\log {\N}+2$ we have that
$\left(\delta+\frac{1}{2}+\varepsilon\right)^{-(j-1)}=\Omega\left(1/(\log
{\N})^{\log_2(\delta+1/2+\epsilon)}\right)$ and so the first of
the additive terms in \eqref{eqn:epsilon_bound} is negligible
for large ${\N}$. Therefore for ${\N}$ sufficiently large, we
have, for all $j=1,\ldots,k-1$
\begin{equation}
\epsilon^2_j \ \leq  \ \frac{1}{2\pi\mu^2 c_1}
\left(\delta+\frac{1}{2}\right)^{-(j-1)}\ .\label{eqn:bound_on_epsilon}
\end{equation}
Since by assumption $\mu>\sqrt{4/c_1}$, we conclude that, for
all ${\N}$ sufficiently large, $\epsilon^2_j<1/(8\pi)$ for all
$j=1,\ldots,k-1$, and so $R_j > \frac{4}{\mu^2}
\left(s_1+(1/2+\varepsilon)^{j-1}z_1\right)$ for
$j=1,\dots,k-1$. Thus, applying Lemma~\ref{lem:gen_result} we
have
\begin{eqnarray*}
  \lefteqn{\Pr(\Gamma)}\\
  &\geq& 1-\sum_{j=1}^{k-1}\exp{\left(\frac{-s_1 \prod_{\ell=1}^{j-1} (1-\epsilon_\ell)}
  {2\sqrt{2 \pi}}\right)} -2\sum_{j=1}^{k-1}\exp{(-2 z_1 (1/2-\varepsilon)^{j-1} \varepsilon^2)}.
\end{eqnarray*}

By a similar argument to that used in Case 1, it is
straightforward to show that
\begin{equation*}
  2\sum_{j=1}^{k-1}\exp{(-2 z_1 (1/2-\varepsilon)^{j-1} \varepsilon^2)} \rightarrow 0 \ .
\end{equation*}
In addition,
\begin{eqnarray*}
  \lefteqn{\sum_{j=1}^{k-1}\exp{\left(\frac{-s_1 \prod_{\ell=1}^{j-1} (1-\epsilon_\ell)}{\sqrt{8
  \pi}}\right)}}\\
  &\leq& (k-1)\exp{\left(\frac{-s_1 \prod_{\ell=1}^{k-2}
  (1-\epsilon_\ell)}{\sqrt{8 \pi}}\right)}\\
  &\leq& (k-1)\exp{\left(\frac{-s_1 \prod_{\ell=1}^{k-2}
  \left(1-\frac{1}{\mu\sqrt{2\pi c_1}} \left(\delta+\frac{1}{2}\right)^{-(\ell-1)/2}\right)}{\sqrt{8 \pi}}\right)}\\
  &\leq& (k-1)\exp{\left(\frac{-s_1 \prod_{\ell=1}^{k-2}
  \left(1-\frac{1}{\sqrt{8\pi}} \left(\delta+\frac{1}{2}\right)^{-(\ell-1)/2}\right)}{\sqrt{8 \pi}}\right)}\ ,
\end{eqnarray*}
where in the last step we used the fact that
$\mu>\sqrt{4/c_1}$. Finally note that from
Lemma~\ref{lem:fractions} we know that
\begin{equation*}
  \prod_{\ell=1}^{k-2} \left(1-\frac{1}{\sqrt{8\pi}}
  \left(\delta+\frac{1}{2}\right)^{-(\ell-1)/2}\right)\rightarrow L(\delta)\ ,
\end{equation*}
where $L(\delta)>0$ hence
\begin{eqnarray}
  \nonumber \lefteqn{\sum_{j=1}^{k-1}
  \exp{\left(\frac{-s_1 \prod_{\ell=1}^{j-1} (1-\epsilon_\ell)}{\sqrt{8\pi}}\right)}}\\
  &\leq& \nonumber (\log_2 \log {\N}+2) \exp{\left(\frac{-{\N}^{1-\beta} (L(\delta)+o(1))}{\sqrt{8
  \pi}}\right)}\\
  &\rightarrow& 0 \ . \label{eqn:Gamma_bnd}
\end{eqnarray}
Therefore we conclude that the event $\Gamma$ happens with probability converging to one.

We  now proceed as before, by conditioning on event $\Gamma$.
The output of the DS procedure consists of a total of
$|I_k|=s_k+z_k$ independent Gaussian measurements with variance
$|I_k|/R_k$, where $s_k$ of them have mean $\mu$ and the
remaining $z_k$ have mean zero. We will show that the proposed
thresholding procedure identifies only true non-zero components
(i.e., correctly rejects all the zero-valued components). In
other words, with probability tending to one,
$\widehat{\mathcal S}_{\rm DS}= {\mathcal S}\cap I_k$. For ease
of notation, and without loss of generality, assume the
$y_{i,k}\sim \mathcal{N}(\mu,|I_k|/R_k)$ for
$i\in\{1,\ldots,s_k\}$ and $y_{i,k}\sim
\mathcal{N}(0,|I_k|/R_k)$ for $i\in\{s_k+1,\ldots,|I_k|\}$.
Then
\begin{eqnarray*}
  \lefteqn{\Pr\left(\left. \widehat{\mathcal S}_{\rm DS}\neq {\mathcal S}\cap I_k \ \right| \ \Gamma\right)}\\
  & = & \Pr\left(\left. \bigcup_{i=1}^{s_k} \left\{y_{i,k}<\sqrt{2/c_k}\right\} \mbox{ or }
  \bigcup_{i=s_k+1}^{|I_k|} \left\{y_{i,k}>\sqrt{2/c_k}\right\} \ \right|
  \ \Gamma\right)\\
  &\leq& s_k \Pr\left({\mathcal N}(\mu,|I_k|/R_k)<\sqrt{2/c_k}\right) + z_k \Pr\left({\mathcal
  N}(0,|I_k|/R_k)>\sqrt{2/c_k}\right).
\end{eqnarray*}
Note that conditioned on the event $\Gamma$ (using arguments similar to those in Case 1)
\begin{eqnarray}
  \nonumber |I_k| & = & s_k+z_k \ \leq \ s_1+z_1\left(\frac{1}{2}+\varepsilon\right)^{k-1} \\
  & \leq & {\N}^{1-\beta} +\frac{{\N}}{2\log {\N}}(1+o(1)) \
  \leq \ \frac{{\N}}{2\log {\N}}(1+o(1))\ . \label{eqn:size_of_I_k}
\end{eqnarray}
Finally, taking into account that $\mu>2\sqrt{2/c_k}$ we
conclude that
\begin{eqnarray*}
  \lefteqn{\Pr\left(\left. \widehat{\mathcal S}_{\rm DS}\neq {\mathcal S}\cap I_k \right|\Gamma\right)}\\
  & \leq & s_k \Pr\left({\mathcal N}(0,|I_k|/R_k)<-\sqrt{2/c_k}\right)+
  z_k \Pr\left({\mathcal N}(0,|I_k|/R_k)>\sqrt{2/c_k}\right)\\
  & \leq & s_k \Pr\left({\mathcal N}(0,1)>\sqrt{\frac{2{\N}}{|I_k|}}\right)+z_k \Pr\left({\mathcal
  N}(0,1)>\sqrt{\frac{2{\N}}{|I_k|}}\right)\\
  & = & |I_k| \Pr\left({\mathcal N}(0,1)>\sqrt{\frac{2{\N}}{|I_k|}}\right)\\
  &\leq& {\N} \Pr \left(\mathcal{N}(0,1)>\sqrt{4\log {\N} (1-o(1))}\right)\\
  &\leq& {\N} \exp\left(-2\log {\N} (1-o(1))\right)\\
  &=& {\N}^{-1+o(1)} \rightarrow 0\ ,
\end{eqnarray*}
where the last inequality follows from the standard Gaussian
tail bound. This together with $\Pr(\Gamma)\rightarrow 1$, and
the fact that $|{\mathcal S}\cap I_k|=s_k =
L(\delta)(1-o(1))s_1$ is bounded away from zero for large
enough $\N$ immediately shows that
$\Pr(\widehat{\mathcal{S}}_{\rm DS}=\emptyset)\rightarrow 0$,
concluding the proof of part (ii) of the theorem.

Part (i) of the theorem follows from the result proved above,
since if $\mu$ is any positive diverging sequence in ${\N}$
then a stronger version of Lemma~\ref{lem:fractions} applies.
In particular, recall \eqref{eqn:bound_on_epsilon}, and note
that Lemma~\ref{lem:fractions} implies
\begin{equation*}
  \prod_{\ell=1}^{k-1} (1-\epsilon_\ell) \ \geq \ \prod_{\ell=1}^{k-1} \left(1-\frac{1}{\mu\sqrt{2\pi c_1}}
\left(\delta+\frac{1}{2}\right)^{-(\ell-1)/2}\right)\rightarrow 1\ .
\end{equation*}
We have already established that the events $\Gamma$ and
$\{\widehat{\mathcal S}_{\rm DS}\neq {\mathcal S}\cap I_k\}$
both hold (simultaneously) with probability tending to one.
Conditionally on these events we have
\begin{equation*}
  \mbox{\rm FDP}(\widehat{\cal S}_{\rm DS})=\frac{0}{s_k}=0\ ,
\end{equation*}
and
\begin{equation*}
  \mbox{\rm NDP}(\widehat{\cal S}_{\rm DS})=\frac{s_1-s_k}{s_1}=1-\frac{s_k}{s_1}\rightarrow 0\ ,
\end{equation*}
since from the definition of $\Gamma$ we have
\begin{equation*}
  s_1 \ \geq \ s_k \ \geq \ s_1 \prod_{\ell=1}^{k-1} (1-\epsilon_\ell) \rightarrow s_1 \ .
\end{equation*}
Therefore we conclude that both $\mbox{\rm FDP}(\widehat{\cal
S}_{\rm DS})$ and $\mbox{\rm NDP}(\widehat{\cal S}_{\rm DS})$
converge in probability to zero as ${\N}\rightarrow \infty$,
concluding the proof of the theorem.

\section{Numerical Experiments}\label{sec:exp}

This section presents numerical experiments with Distilled
Sensing (DS).  The results demonstrate that the asymptotic
analysis predicts the performance in finite dimensional cases
quite well.  Furthermore, the experiments suggest useful rules
of thumb for implementing DS in practice.

There are two input parameters to the DS procedure; the number
of distillation steps, $k$, and the distribution of precision
across the steps, $\{R_j\}_{j=1}^k$.  Throughout our
simulations we choose $k=\max\{\lceil \log_2 \log
{\N}\rceil,0\}+2$, as prescribed in Theorem~\ref{thm:ds}. For
the precision distribution, first recall the discussion
following the proof of Lemma~\ref{lem:gen_result}.  There it is
argued that if the sparsity model is valid, a sufficient
condition for the precision distribution is $R_j > R_1
(1/2+\varepsilon)^{j-1}$, $j=1,\dots,k$, with $0 < \varepsilon
< 1/2$.  In words, the precision allocated to each step must be
greater than $1/2$ the precision allocated in the previous
step.  In practice, we find that choosing $R_{j+1}/R_{j} =
0.75$ for $j=1,\dots,k-2$ provides good performance over the
full SNR range of interest. Also, from the proof of the main
result (Theorem~\ref{thm:ds}) we see that the threshold for
detection is inversely proportional to the square root of the
precision allocation in the first and last steps. Thus, we have
found that allocating equal precision in the first and last
steps is beneficial. The intuition is that the first step is
the most crucial in controlling the NDP and the final step is
most crucial in controlling the FDP. Thus, the precision
allocation used throughout the simulations follows this simple
formula:
\begin{eqnarray*}
R_j & = &  (0.75)^{j-1} R_1 \, , \ j=2,\dots,k-1 \, , \\
R_k & = & R_1 \, ,
\end{eqnarray*}
and $R_1$ is chosen so that  $\sum_{j=1}^k R_j ={\N}$.

Figure~\ref{fig:DS_results} compares the FDP vs.\ NDP
performance of the DS procedure to non-adaptive (single
observation) measurement at several signal-to-noise ratios (SNR
= $\mu^2$). We consider signals of length ${\N}=2^{14}$ having
$\sqrt{\N}=128$ non-zero components with uniform amplitude with
locations chosen uniformly at random. This choice of signal
dimension corresponds to $k=6$ observation steps in the DS
procedure. The range of FDP-NDP operating points is surveyed by
varying the threshold applied to the non-adaptive measurements
and the output of the DS procedure for each of $1000$ trials,
corresponding to different realizations of randomly-generated
signal and additive noise. Recall that largest squared
magnitude in a realization of $\N$ i.i.d.\ ${\cal N}(0,1)$
variables grows like $2 \log {\N}$, and in our experiment, $2
\log {\N} \approx 20$. Consequently, when the $\mbox{SNR}=20$
we see that both DS and non-adaptive measurements are highly
successful, as expected. Another SNR level of interest is $8$,
since in this case this happens to approximately satisfy the
condition $\mu = \sqrt{2/c_1} = \sqrt{2{\N}/R_1}$, which
according to the Theorem~\ref{thm:ds} is a critical level for
detection using DS. The simulations show that DS remains highly
successful at this level while the non-adaptive results are
poor. Finally, when the $\mbox{SNR}=2$, we see that DS still
yields useful results. For example, at $\mbox{FDP}= 0.05$, the
DS procedure has an average NDP of roughly $80\%$ (i.e., $20\%$
of the true components are still detected, on average).  This
demonstrates the approximate $\log {\N}$ extension of the SNR
range provided by DS.  Note the gap in the FDP values of the DS
results (roughly from $0.75$ to $1$). The gap arises because
the the output of DS has a higher SNR and is much less sparse
than the original signal, and so arbitrarily large FDP values
cannot be achieved by any choice of threshold. Large FDP values
are, of course, of little interest in practice.  We also remark
on the structured patterns observed in cases of high NDP and
low FDP (in upper left of figures for $\mbox{SNR}=2$ and
$\mbox{SNR}=8$). The visually structured `curves' of NDP-FDP
pairs arise when the total number of discoveries is small, and
hence the FDP values are restricted to certain rational
numbers.  For example, if just $3$ components are discovered,
then the number of false-discoveries can only take the values
$0,$ $1/3,$ $2/3,$ and $1.$

\begin{figure}[t]
\centering
\subfigure[]{
\includegraphics[width = 6.5cm]{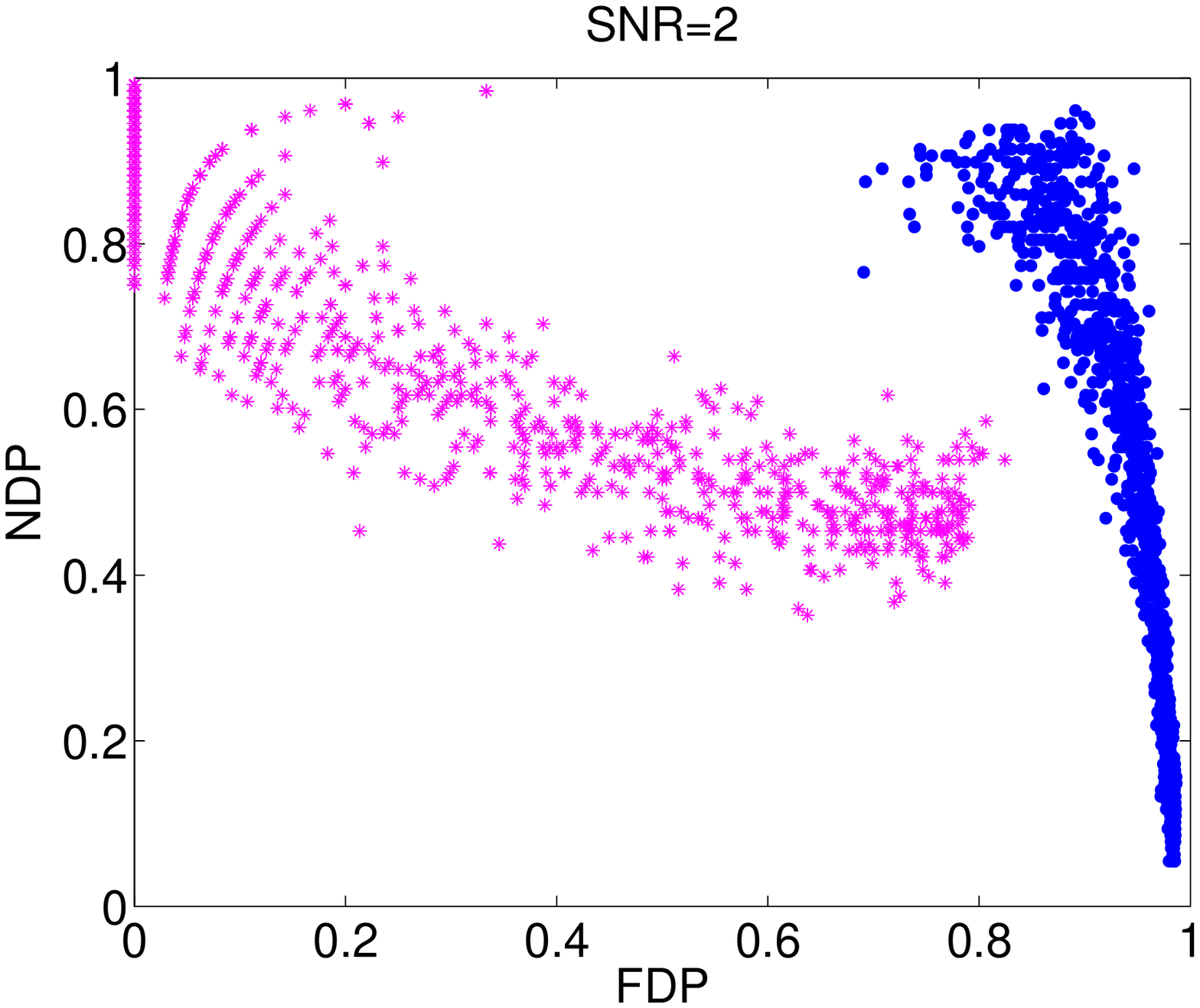}
}\hspace{1em}
\subfigure[]{
\includegraphics[width = 6.5cm]{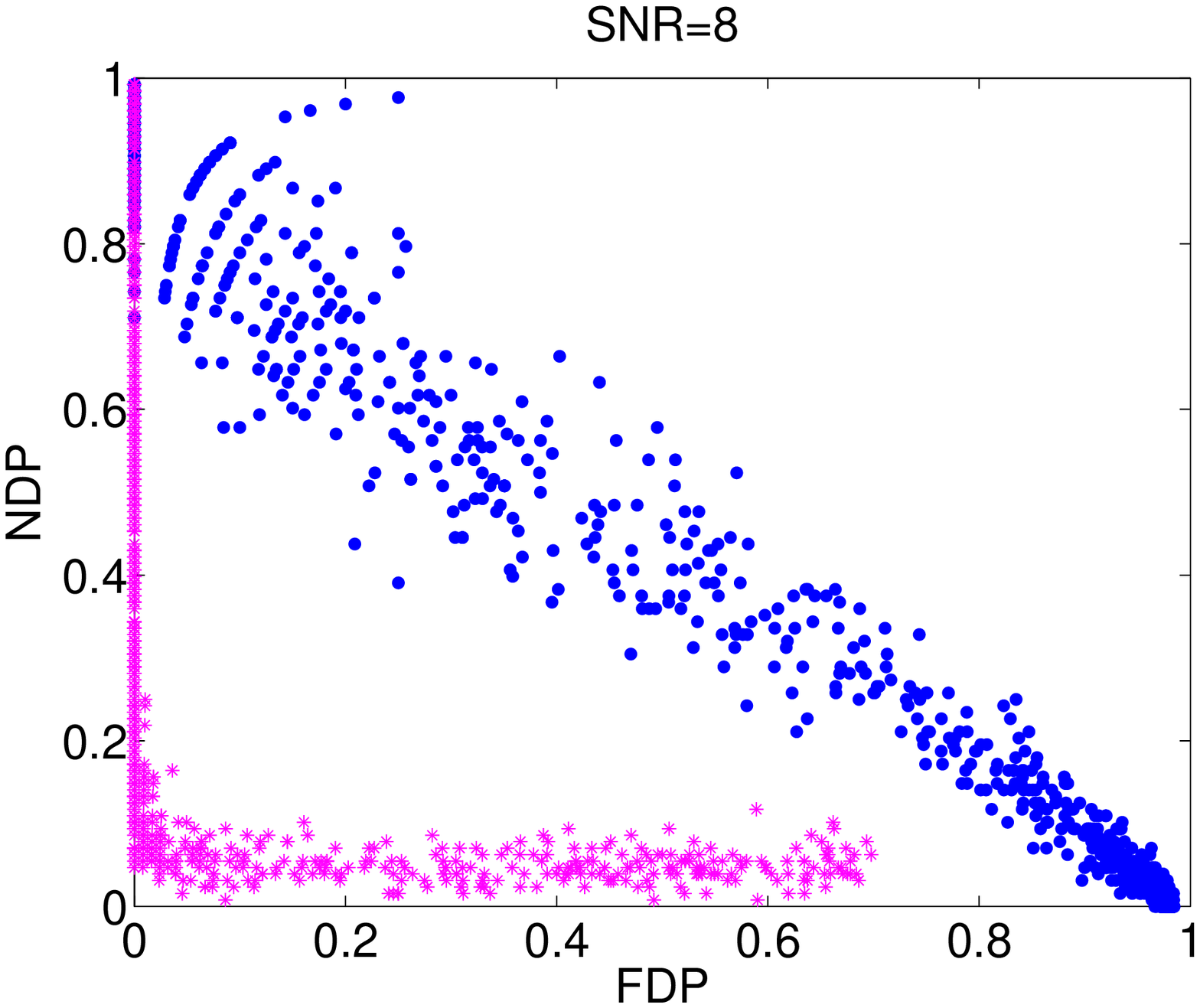}
}\\
\subfigure[]{
\includegraphics[width = 6.5cm]{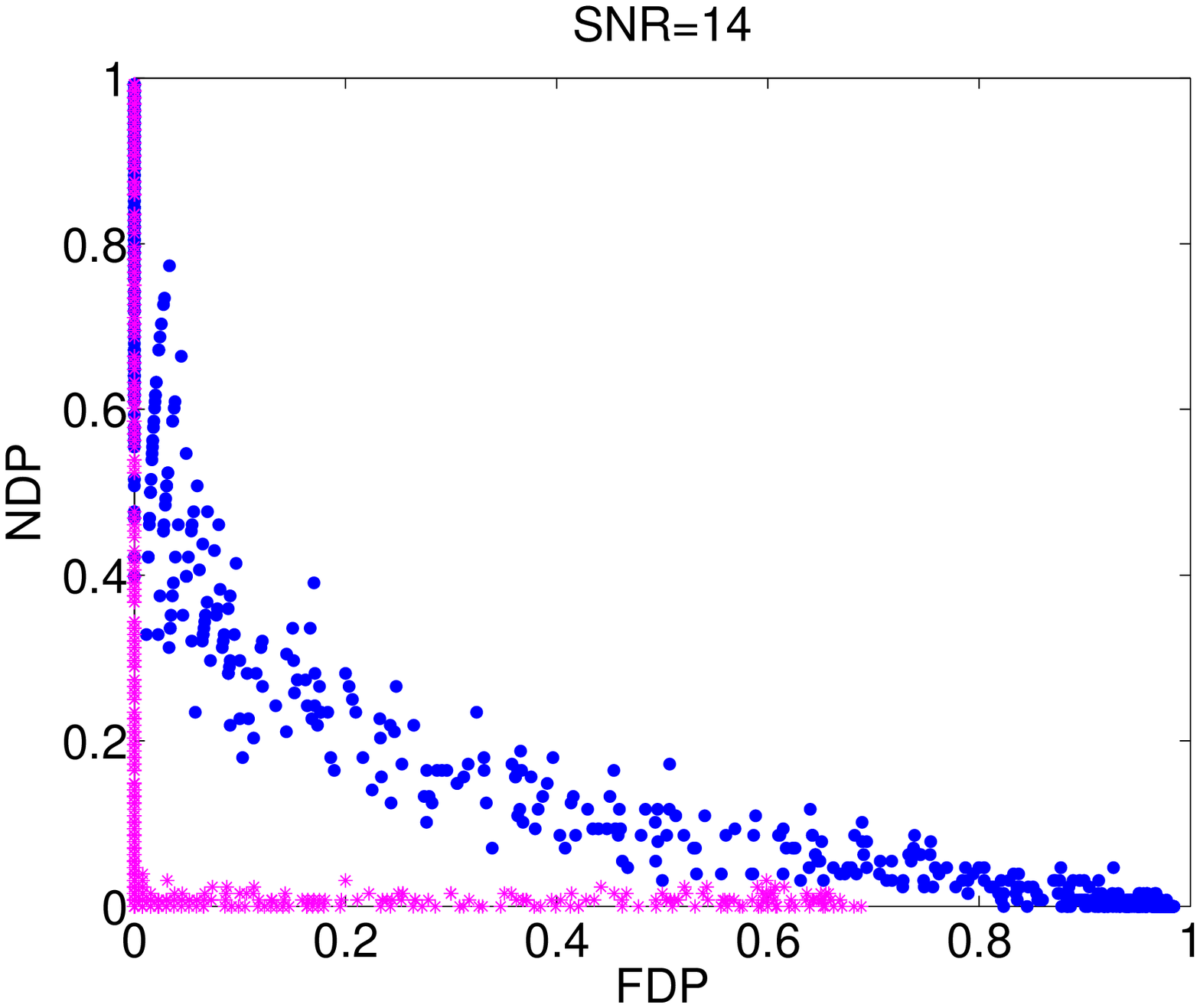}
}\hspace{1em}
\subfigure[]{
\includegraphics[width = 6.5cm]{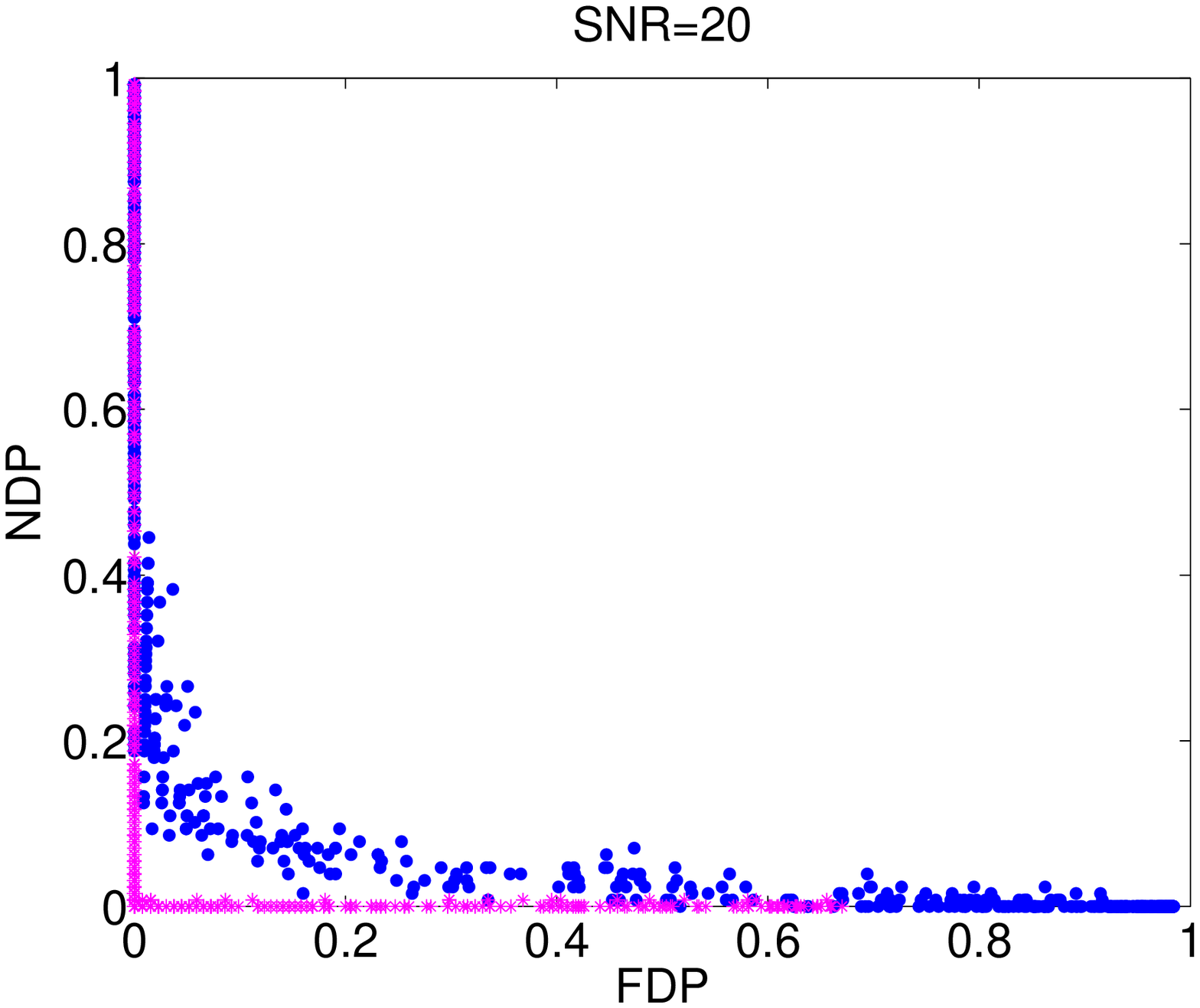}
}
\caption{FDP and NDP performance for DS (indicated with \textcolor{magenta}{$\ast$})
and non-adaptive sensing (indicated with \textcolor{blue}{$\bullet$}) at different SNRs.
Smaller values of FDP and NDP correspond to more accurate recovery (ie, exact support recovery
occurs when NDP $=$ FDP $=0$). The results clearly show that DS outperforms non-adaptive sensing
for each SNR examined.}
\label{fig:DS_results}
\end{figure}

Figure~\ref{fig:DS_FDRvSNR} compares the performance of
non-adaptive sensing and the DS procedure in terms of the
false-discovery rate (FDR) and the non-discovery rate (NDR),
which are the average FDP and NDP, respectively. We consider
three different cases, corresponding to signals of length $\N =
2^{14}, 2^{17},$ and $2^{20}$, (the solid, dashed, and dash-dot
lines, respectively) where for each case the number of non-zero
signal components is $\lfloor \N^{1/2} \rfloor$. The precision
allocation and number of observation steps are chosen as
described above (here, $k=6$ for each of the three cases). For
each value of SNR, $500$ independent experiments were performed
for DS and non-adaptive sampling, and in each, thresholds were
selected so that the FDRs were fixed at approximately $0.05$.
The resulting average FDRs and NDRs for each SNR level are
shown. The results show that not only does DS achieve
significantly lower NDRs than non-adaptive sampling over the
entire SNR range, its performance also exhibits much less
dependence on the signal dimension $\N$.

\begin{figure}[t]
\centering
\includegraphics[width=9cm]{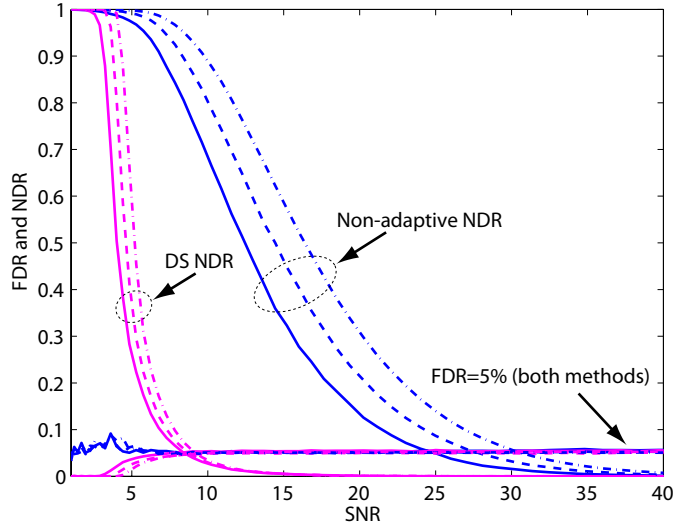}
\caption{FDR and NDR vs. SNR comparison. The solid, dashed, and dash-dot lines
correspond to signals of length $\N=2^{14}, 2^{17}$, and $2^{20}$, respectively, having
$\lfloor \N^{1/2} \rfloor$ non-zero entries. At each value of SNR and for each method
(DS and non-adaptive sampling), thresholds were selected to achieve FDR $= 0.05$. Lower
values of NDR correspond to more accurate recovery; DS clearly outperforms non-adaptive sensing
over the entire SNR range and shows much less dependence on the signal dimension $\N$.}
\vspace{-1em}
\label{fig:DS_FDRvSNR}
\end{figure}

\section{Concluding Remarks}\label{sec:conclusion}

There has been a tremendous interest in high-dimensional
testing and detection problems in recent years.  A
well-developed theory exists for such problems when using a
single, non-adaptive observation model
\cite{abramovich:06,donoho:06,ingster:97,ingster:98,donoho:04}.
However, in practice and theory, multistage adaptive designs
have shown promise
\cite{muller:07,zehetmayer:05,satagopan:03,zehetmayer:08}. This
paper quantifies the improvements such methods can achieve.  We
proposed and analyzed a specific multistage design called
Distilled Sensing (DS), and established that DS is capable of
detecting and localizing much weaker sparse signals than
non-adaptive methods.  The main result shows that adaptivity
allows reliable detection and localization at a signal-to-noise
ratio (SNR) that is roughly $\log \N$ lower than the minimum
required by non-adaptive methods, where $\N$ is the problem
dimension.  To put this in context, suppose one is interested
in screening $\N=20,000$ genes, then $\log \N \approx 10$.
Thus, the gains can be quite significant in problem sizes of
practical interest, which is why experimentalists often do
employ similar methods.

An additional point worthy of future investigation is the
development of lower bounds, characterizing the minimum
amplitude $\mu(\N)$ below which signal detection and
localization are impossible for \emph{any} sensing procedure
(including adaptive sensing). In general, lower bounds are
difficult to devise for sequential experimental design
settings, with a few notable exceptions \cite{korostelev:99,
castro:08}. Here, our results establish that significant
improvements are achievable using adaptivity, although we
relegate any general claims of optimality for adaptive sensing
procedures to future work.

There are several possible extensions to DS. One is to consider
even sparser signal models, where the number of nonzero entries
is significantly smaller than $\N^{1-\beta}$ for
$\beta\in(0,1)$, as considered here. In particular, the same
asymptotic results stated here follow also for signals whose
sparsity levels are as small as a constant times
$\log\log\log{\N}$. Indeed, making this choice of $s_1$ in
\eqref{eqn:epsilon_bound} leads to the same bound on the
$\epsilon_j^2$ given in \eqref{eqn:bound_on_epsilon}, and this
choice is also sufficient to ensure that \eqref{eqn:Gamma_bnd}
holds as well. In addition, for this choice of $s_1$ the same
bound is obtained in \eqref{eqn:size_of_I_k}, and the rest of
the proof goes through as stated. Another extension is to use
DS with alternate measurement models. For example, each
measurement could be a linear combination of the entries of
$x$, rather than direct measurements of individual components.
If the linear combinations are non-adaptive, this leads to a
regression model commonly studied in the Lasso and Compressed
Sensing literature---see, for example, \cite{tibshirani:96,
candes:06}. However, sequentially tuning the linear
combinations leads to an adaptive version of the regression
model which can be shown to provide significant improvements,
as well \cite{haupt:09b}.

\appendices

\section{Thresholds for Non-Adaptive Recovery}\label{sec:naloc}

In this section we give a proof of Theorem~\ref{thm:naloc}. We
will proceed by considering two cases separately: (i) $r>\beta$
and (ii) $r < \beta$. The analysis of the phase transition
point $r=\beta$ is interesting, but it is beyond the scope of
this paper. Begin by noticing that in the setting of the
theorem the minimax optimal support estimation procedure to
control the false and non-discovery proportions is a simple
coordinate-wise thresholding procedure of the form
\begin{equation*}
\widehat{\cal S}=\{i:y_i>\tau\}\ ,
\end{equation*}
where $\tau\geq 0$ can be chosen appropriately.  A formal proof
of this optimality can be done by noting that the class of
hypothesis is invariant under permutations (see
\cite{ingster:97,ingster:98} for details).

\textbf{Case (i) $r>\beta$:} In this case the signal support
can be accurately identified from the observations, in the
sense that $\mathrm{FDP}(\widehat{\cal S})$ and
$\mathrm{NDP}(\widehat{\cal S})$ both converge in probability
to zero. For this case we will take $\tau =
\tau({\N})=\sqrt{2\alpha\log {\N}}$, where $\beta<\alpha<r$.

Begin by defining $D_z$ and $M_s$ to be the number of retained
non-signal components and the number of missed signal
components, respectively. Formally
\begin{equation*}
  D_z=\sum_{i=1}^{\N} \1_{\{y_i>\tau, \ x_i=0\}}\ ,
\end{equation*}
and
\begin{equation*}
  M_s=\sum_{i=1}^{\N} \1_{\{y_i\leq\tau, \ x_i\neq 0\}}\ .
\end{equation*}
Note that $D_z$ is binomially distributed, that is $D_z \sim
\mathrm{Bin}({\N}(1-{\N}^{-\beta}),q_z)$, where
$q_z=\Pr(y_i>\tau)$ when $i$ is such that $x_i=0$. By noticing
that $\tau>0$ and using a standard Gaussian tail bound we have
that
\begin{equation*}
  q_z\leq \frac{1}{\sqrt{2\pi\tau^2}} \exp\left(-\frac{\tau^2}{2}\right)=\frac{1}{\sqrt{4\pi\alpha \log {\N}}}
  {\N}^{-\alpha} .
\end{equation*}

In a similar fashion note that $M_s \sim
\mathrm{Bin}({\N}^{1-\beta},q_s)$, where $q_s=\Pr(y_i\leq
\tau)$ when $i$ is such that $x_i=\sqrt{2r\log {\N}}$. Let
$Z\sim \mathcal{N}(0,1)$ be an auxiliary random variable. Then
\begin{eqnarray*}
  q_s&=&\Pr(Z+\sqrt{2r\log {\N}} \leq \tau)\\
  &=&\Pr(Z \leq \tau-\sqrt{2r\log {\N}})\\
  &=&\Pr(Z > \sqrt{2\log {\N}}(\sqrt{r}-\sqrt{\alpha}))\ ,
\end{eqnarray*}
And so, using the Gaussian tail bound we have
\begin{equation*}
  q_s \leq \frac{1}{\sqrt{4\pi \log {\N}}(\sqrt{r}-\sqrt{\alpha})} {\N}^{-(\sqrt{r}-\sqrt{\alpha})^2}\ .
\end{equation*}

We are ready to show that both $\mathrm{FDP}(\widehat{\cal S})$
and $\mathrm{NDP}(\widehat{\cal S})$ converge in probability to
zero. Begin by noticing that $\mathrm{NDP}(\widehat{\cal
S})=M_s/{\N}^{1-\beta}$. By definition
$\mathrm{NDP}(\widehat{\cal S}) \stackrel{\tiny P}{\rightarrow}
0$  means that for all fixed $\epsilon>0$,
\begin{equation*}
  \Pr(|\mathrm{NDP}(\widehat{\cal S})|>\epsilon)\rightarrow 0\ ,
\end{equation*}
as ${\N} \rightarrow \infty$. Noting that
$\mathrm{NDP}(\widehat{\cal S})$ is non-negative, this can be
easily established using Markov's inequality.
\begin{eqnarray*}
  \Pr(\mathrm{NDP}(\widehat{\cal S})>\epsilon)&=&\Pr\left(\frac{M_s}{{\N}^{1-\beta}}\geq \epsilon\right)\\
  &=&\Pr(M_s> \epsilon {\N}^{1-\beta})\\
  &\leq& \frac{\E[M_s]}{\epsilon {\N}^{1-\beta}}\\
  &=&\frac{{\N}^{1-\beta}q_s}{\epsilon {\N}^{1-\beta}}=\frac{q_s}{\epsilon}\rightarrow 0\ ,
\end{eqnarray*}
as ${\N}\rightarrow\infty$ as clearly $q_s$ converges to zero
(since $r>\alpha$). For the false discovery proportion the
reasoning is similar. Note that the number of correct
discoveries is ${\N}^{1-\beta}-M_s$. Taking this into account
we have
\begin{equation*}
  \mathrm{FDP}(\widehat{\cal S})=\frac{D_z}{{\N}^{1-\beta}-M_s+D_z}\ .
\end{equation*}
Let $\epsilon>0$. Then
\begin{eqnarray*}
  \lefteqn{ \Pr(\mathrm{FDP}(\widehat{\cal S})>\epsilon)}&&\\
  &=& \Pr\left(\frac{D_z}{{\N}^{1-\beta}-M_s+D_z}> \epsilon\right)\\
  %&=&\Pr\left(D_z > \epsilon (n^{1-\beta}-M_s+D_z)\right)\\
  %&=&\Pr\left((1-\epsilon)D_z + \epsilon M_s > \epsilon n^{1-\beta}\right)\\
  &=&\Pr\left((1-\epsilon)\frac{D_z}{{\N}^{1-\beta}} + \epsilon \frac{M_s}{{\N}^{1-\beta}} > \epsilon\right)\\
  &\leq&\frac{\E\left[(1-\epsilon)\frac{D_z}{{\N}^{1-\beta}} + \epsilon \frac{M_s}{{\N}^{1-\beta}}\right]}{\epsilon}\\
  &=&\frac{1-\epsilon}{\epsilon}\frac{{\N}(1-{\N}^{-\beta})q_z}{{\N}^{1-\beta}} + q_s\\
  &\leq&\frac{1-\epsilon}{\epsilon}{\N}^\beta(1-{\N}^{-\beta}) \frac{1}{\sqrt{4\pi\alpha \log {\N}}} {\N}^{-\alpha}
  + \frac{1}{\sqrt{4\pi \log {\N}}(\sqrt{r}-\sqrt{\alpha})} {\N}^{-(\sqrt{r}-\sqrt{\alpha})^2}\ ,
\end{eqnarray*}
where the last line clearly converges to zero as
${\N}\rightarrow \infty$, since $\beta <\alpha < r$. Therefore
we conclude that $\mathrm{FDP}(\widehat{\cal S})$ converges to
zero in probability.

\textbf{Case (ii) $r < \beta$:} In this case we will show that
no thresholding procedure can simultaneously control the false
and non-discovery proportions. Begin by noting that the smaller
$\tau$ is, the easier it is to control the non-discovery
proportion. In what follows we will identify an upper-bound on
$\tau$ necessary for the control of the non-discovery rate.
Note that if $\tau=\tau({\N})=\sqrt{2r\log {\N}}$ then
$q_s=1/2$, and therefore
\begin{equation*}
  \mathrm{NDP}(\widehat{\cal S})=\frac{M_s}{\N^{1-\beta}}\stackrel{\tiny a.s.}{\rightarrow} 1/2\ ,
\end{equation*}
as ${\N}\rightarrow\infty$, by the law of large numbers.
Therefore a necessary, albeit insufficient, condition for
$\mathrm{NDP}(\widehat{\cal S}) \stackrel{\tiny P}{\rightarrow}
0$ is that for all but finitely many ${\N}$
\begin{equation}\label{eqn:upper_bound_tau}
  \tau < \sqrt{2r\log {\N}}\ .
\end{equation}

Similarly, note that the larger $\tau$ is, the easier it is to
control the false discovery rate. In the same spirit of the
above derivation we will identify a lower-bound for $\tau$ that
must necessarily hold in order to control the false-discovery
rate. Recall the previous derivation, where we showed that, for
any $\epsilon>0$
\begin{eqnarray*}
  \Pr(\mathrm{FDP}(\widehat{\cal S})>\epsilon)&=&\Pr\left((1-\epsilon)\frac{D_z}{{\N}^{1-\beta}} + \epsilon
  \frac{M_s}{{\N}^{1-\beta}} \geq \epsilon\right)\\
  &\geq&\Pr\left((1-\epsilon)\frac{D_z}{{\N}^{1-\beta}} \geq \epsilon\right)\\
  &=&\Pr\left(\frac{D_z}{{\N}^{1-\beta}} \geq \frac{\epsilon}{1-\epsilon}\right)\ ,
\end{eqnarray*}
where the last inequality follows trivially given that $M_s\geq
0$ and, without loss of generality, we assume that
$\epsilon<1$. This means that $\mathrm{FDP}(\widehat{\cal S})$
converges in probability to zero only if
$\frac{D_z}{{\N}^{1-\beta}}$ also converges in probability to
zero. Namely, for any $\epsilon>0$ we must have
$\lim_{{\N}\rightarrow \infty}
\Pr(D_z/{\N}^{1-\beta}>\epsilon)=0$. In what follows take
$\tau=\sqrt{2r\log {\N}}$. Let $\epsilon>0$ and note that
\begin{eqnarray*}
  \Pr\left(\frac{D_z}{{\N}^{1-\beta}}>\epsilon\right)&=&\Pr(D_z>\epsilon {\N}^{1-\beta})\\
  &=&\Pr(D_z-\E[D_z]>\epsilon {\N}^{1-\beta}-\E[D_z])\\
  &=&\Pr(D_z-\E[D_z]>\epsilon {\N}^{1-\beta}-{\N}(1-{\N}^{-\beta})q_z)\ .
\end{eqnarray*}
Define $a=\epsilon {\N}^{1-\beta}-{\N}(1-{\N}^{-\beta})q_z$.
Note that by the Gaussian tail bound, we have
\begin{equation*}
  \frac{1}{\sqrt{4\pi r \log {\N}}} \left(1-\frac{1}{2r\log {\N}}\right){\N}^{-r}
  \leq q_z \leq \frac{1}{\sqrt{4\pi r \log {\N}}} {\N}^{-r}\ ,
\end{equation*}
or equivalently,
\begin{equation*}
  q_z = \frac{1-o(1)}{\sqrt{4\pi r \log {\N}}}  {\N}^{-r} .
\end{equation*}
Given this it is straightforward to see that
\begin{eqnarray*}
  a &=& \epsilon {\N}^{1-\beta}-(1-o(1))\frac{{\N}(1-{\N}^{-\beta})}{\sqrt{4\pi r \log {\N}}}  {\N}^{-r}\\
  &=& \epsilon {\N}^{1-\beta}-(1-o(1))\frac{{\N}^{1-r}}{\sqrt{4\pi r \log {\N}}} \\
  &=& {\N}^{1-r}\left(\epsilon {\N}^{r-\beta}-\frac{1-o(1)}{\sqrt{4\pi r \log {\N}}}\right) \\
  &=& -\frac{{\N}^{1-r}}{\sqrt{4\pi r \log {\N}}}(1-o(1))\ ,
\end{eqnarray*}
where in the last step we use the assumption that $\beta>r$.
Therefore $a \rightarrow -\infty$ as ${\N}$ goes to infinity.
Let ${\N}_0(\epsilon)\in\mathbb{N}$ be such that $a <0$ for all
${\N}\geq {\N}_0(\epsilon)$. Then
\begin{eqnarray*}
  \Pr(D_z/{\N}^{1-\beta}>\epsilon)&=&\Pr(D_z-\E[D_z]>a)\\
  &=& 1-\Pr(D_z-\E[D_z]\leq a)\\
  &\geq& 1-\Pr(|D_z-\E[D_z]|\geq -a)\\
  &\geq& 1-\frac{\mathrm{Var}(D_z)}{(-a)^2}\ ,
\end{eqnarray*}
where $\mathrm{Var}(D_z)={\N}(1-{\N}^{-\beta})q_z(1-q_z)$ is
the variance of $D_z$ and the last step uses Chebyshev's
inequality. Recalling that ${\N}\geq {\N}_0(\epsilon)$ we can
examine the last term in the above expression easily.
\begin{eqnarray*}
  1-\frac{\mathrm{Var}(D_z)}{(-a)^2} &=& 1-(1-q_z)\frac{{\N}(1-{\N}^{-\beta})q_z}{a^2}\\
  &=& 1-(1-o(1))\frac{{\N} q_z}{a^2}\\
  &=& 1-(1-o(1))\frac{{\N}^{1-r}}{\sqrt{4\pi r \log {\N}}}\frac{4\pi r \log {\N}}{{\N}^{2-2r}}\\
  &=& 1-(1-o(1))\frac{\sqrt{4\pi r \log {\N}}}{{\N}^{1-r}}\rightarrow 1\ ,
\end{eqnarray*}
as ${\N}\rightarrow\infty$. Therefore we conclude that, for
$\tau=\sqrt{2r\log {\N}}$, $D_z/{\N}^{1-\beta}$ does not
converge in probability to zero, and therefore
$\mathrm{FDP}(\widehat{\cal S})$ also does not converge to
zero.

The above result means that a necessary condition for the
convergence of $\mathrm{FDP}(\widehat{\cal S})$ to zero is that
for all but finitely many ${\N}$
\begin{equation*}
  \tau > \sqrt{2r\log {\N}}\ .
\end{equation*}
This, together with \eqref{eqn:upper_bound_tau} shows that
there is no thresholding procedure capable of controlling both
the false-discovery and non-discovery proportions when $r <
\beta$ as we wanted to show, concluding the proof.

\section{Auxiliary Material}

\begin{lemma}\label{lem:lim}
  Let $0\leq f({\N}) \leq 1/2$ and $g({\N})\geq 0$ be any
  sequences in ${\N}$ such that $\lim_{{\N}\rightarrow\infty}
  f({\N}) g({\N}) = 0$. Then
  \begin{equation*}
    \lim_{{\N}\rightarrow\infty} \left(1+ f({\N})\right)^{g({\N})} =
    \lim_{{\N}\rightarrow\infty} \left(1- f({\N})\right)^{g({\N})}=1 \ .
  \end{equation*}
\end{lemma}

\begin{proof}
  To establish that
  $\lim_{{\N}\rightarrow\infty}(1+f({\N}))^{g({\N})}=1$ note
  that
  \begin{equation*}
    1\leq (1+f({\N}))^{g({\N})}=\exp\left(g({\N})
    \log(1+f({\N}))\right)\leq \exp\left(g({\N})f({\N})\right)\ ,
  \end{equation*}
  where the last inequality follows from $\log(1+x) \leq x$ for
  all $x\geq 0$. As $g({\N})f({\N})\rightarrow 0$ we conclude
  that $\lim_{{\N}\rightarrow\infty}(1+f({\N}))^{g({\N})}=1$.

  The second part of the result is established in a similar
  fashion. Note that
  \begin{eqnarray*}
    \log\left(1-f({\N})\right)=
    -\log{\left(\frac{1}{1-f({\N})}\right)}
    &=&\log{\left(1+\frac{f({\N})}{1-f({\N})}\right)}\\
    &\geq& -\frac{f({\N})}{1-f({\N})} \geq -2f({\N}) \,
  \end{eqnarray*}
  where the last inequality relies on the fact that
  $f({\N})\leq 1/2$. Using this fact we have that
  \begin{equation*}
    1\geq (1-f({\N}))^{g({\N})}=
    \exp\left(g({\N})\log(1-f({\N}))\right)\geq \exp\left(-2f({\N})g({\N})\right)\ .
  \end{equation*}
  Taking into account that $g({\N})f({\N})\rightarrow 0$
  establishes the desired result.
\end{proof}

\begin{lemma}\label{lem:fractions}
  Let $k=k({\N})$ be a positive integer sequence in ${\N}$,
  and let $g=g({\N})$ be a positive nondecreasing sequence in
  ${\N}$. For some fixed $a>1$ let
  $\epsilon_j=\epsilon_j({\N}) \leq a^{-j}/g({\N})$. If
  $g({\N}) > a^{-1}(1 + \eta)$, for some fixed $\eta>0$, then
  \begin{equation*}
    \lim_{{\N}\rightarrow \infty} \prod_{j=1}^{k({\N})}
    \left(1-\epsilon_j({\N})\right) > 0.
  \end{equation*}
  If, in addition, $g({\N})$ is any positive monotone diverging sequence in ${\N}$, then
  \begin{equation*}
    \lim_{{\N}\rightarrow \infty} \prod_{j=1}^{k({\N})}
    \left(1-\epsilon_j({\N})\right) = 1.
  \end{equation*}
\end{lemma}

\begin{proof}
  Note that
  \begin{eqnarray*}
    \log{\left(\prod_{j=1}^{k({\N})} \left(1-\epsilon_j({\N})\right)\right)}
    &\geq& \sum_{j=1}^{k({\N})} \log{\left(1-\frac{a^{-j}}{g({\N})}\right)}\\
    &=& -\sum_{j=1}^{k({\N})} \log{\left(1+\frac{a^{-j}/g({\N})}{1-a^{-j}/g({\N})}\right)}\\
    &\geq& -\sum_{j=1}^{k({\N})} \frac{a^{-j}/g({\N})}{1-a^{-j}/g({\N})}\\
    &\geq& \frac{-1}{1-a^{-1}/g({\N})}\sum_{j=1}^{k({\N})} a^{-j}/g({\N})\\
    &=& \frac{-1}{g({\N})-a^{-1}}\sum_{j=1}^{k({\N})} a^{-j}.
  \end{eqnarray*}
  Now, using the formula for the sum of a geometric series, we
  have
  \begin{equation*}
    \log{\left(\prod_{j=1}^{k({\N})}
    \left(1-\epsilon_j({\N})\right)\right)} \geq \frac{-1}{g({\N})-a^{-1}}
    \left[\frac{a^{-1}(1-a^{-k({\N})})}{1-a^{-1}}\right],
  \end{equation*}
  from which it follows that
  \begin{equation*}
    \prod_{j=1}^{k({\N})}\left(1-\epsilon_j({\N})\right) \geq \exp{\left(\frac{-1}{g({\N})-a^{-1}} \left[\frac{a^{-1}(1-a^{-k({\N})})}{1-a^{-1}}\right]\right)} \ .
  \end{equation*}

  Now, assuming only that $g({\N}) > a^{-1}(1 +\eta)$, for some fixed $\eta>0$ it is easy
  to see that
  \begin{equation*}
    \lim_{{\N}\rightarrow \infty}\prod_{j=1}^{k({\N})}\left(1-\epsilon_j({\N})\right) > 0,
  \end{equation*}
  and if $g({\N})\rightarrow \infty$ as $\N\rightarrow\infty$ we have
  \begin{equation*}
    \lim_{{\N}\rightarrow \infty} \prod_{j=1}^{k({\N})}\left(1-\epsilon_j({\N})\right) = 1,
  \end{equation*}
  as claimed.
\end{proof}

\bibliographystyle{IEEEtran}
\bibliography{DS_rev_bib}

\end{document}